\documentclass[12pt, reqno]{amsart}
\usepackage{amsmath, amsthm, amsfonts, amssymb, verbatim, graphicx, enumerate}
\usepackage{a4wide}

\usepackage{tikz}
\usetikzlibrary{calc,decorations.markings}

\newcommand{\R}{\mathbb{R}}
\newcommand{\Q}{\mathbb{Q}}
\newcommand{\N}{\mathbb{N}}
\newcommand{\C}{\mathbb{C}}
\newcommand{\Z}{\mathbb{Z}}
\newcommand{\uhp}{\mathbb{H}}

\newcommand{\im}{\textnormal{Im}}

\newtheorem{thm}{Theorem}
\newtheorem*{thm*}{Theorem Ch}
\newtheorem{cor}{Corollary}
\newtheorem{lem}{Lemma}
\newtheorem{prop}{Proposition}
\newtheorem*{prop*}{Proposition J}

{\theoremstyle{definition}
\newtheorem{claim}{Claim}

\newtheorem{rmk}{Remark}

}

\title[]{H\"{o}lder regularity of arithmetic Fourier series arising from modular forms}
\author[]{Izabela Petrykiewicz}

\address[I. Petrykiewicz]{Universit\'e Joseph Fourier, Institut Fourier, 100 rue des maths, 38402 St Martin d'H\`{e}res, France
}
\email{izabela.petrykiewicz@ujf-grenoble.fr}



\date{\today}

\begin{document}
 
\begin{abstract}
Given a modular form which is not a cusp form $M_k(z)=\sum_{n=0}^{\infty}r_ne^{2\pi inz}$ of weight $k \geq 4$, we define the series 
$M_{k,s}(x)=\sum_{n=1}^{\infty}\frac{r_n}{n^s}\sin(2\pi nx),$ which converges for all $x\in\R$ when $s>k$. In this paper, we compute 
the H\"{o}lder regularity exponent of $M_{k,s}$ at irrational points. In our analysis we apply wavelets methods proposed by Jaffard 
in 1996 in the study of the Riemann series. We find that the H\"{o}lder regularity exponent at a point $x$ is related to the fine 
diophantine properties of $x$, in a very precise way. 
\end{abstract}

\maketitle

\section{Introduction and statement of the results}

In this paper, we study the H\"{o}lder regularity exponent of certain trigonometric series related to modular forms. 
We say that $f\in C^{\alpha}(x_0)$ for some $\alpha>0$ when there exists a polynomial $P$ of degree less than or equal to $[\alpha]$, 
and a constant $C$ such that 
$$|f(x)-P(x-x_0)|\leq C|x-x_0|^{\alpha},$$
as $x\to x_0$. Then we define the H\"{o}lder regularity exponent of $f$ at $x_0$ as $\alpha(x_0)=\sup\{\beta:f\in C^{\beta}(x_0)\}$.

Let $k\geq 4$ be even, and let $$M_k(z)=\sum_{n=0}^{\infty}r_ne^{2\pi inz}$$ be a modular form under $SL_2(\Z)$ of weight $k$, defined over
 $\uhp=\{z\in\C|\im(z)>0\}$; it is a cusp form when $r_0=0$. We then consider the series
$$M_{k,s}(x)=\sum_{n=1}^{\infty}\frac{r_n}{n^s}\sin(2\pi nx),$$
for suitable $s\in\R$ and $x\in\R$. We are interested in the H\"{o}lder regularity exponent of $M_{k,s}$ at $x\in\R\setminus\Q$. 

\medskip
This work is motivated by the example of the Riemann ``non-differentiable'' function which is defined as
\begin{equation}\label{riems}
 S(x)=\sum_{n=1}^{\infty}\frac{1}{n^2}\sin(\pi n^2 x).
\end{equation}
This kind of series were first introduced by Riemann and also studied by Chowla and Walfisz \cite{ChoWa}, see also \cite{LMZ}.
The differentiability and pointwise H\"{o}lder regularity of $S$ have been studied for about 80 years by many mathematicians like Hardy, Littlewood, Gerver, 
Itatsu, Duistermaat and Jaffard, see \cite{H, HL, G, I, Du, J1, J2}. The function $f$ is only differentiable at rational points of the form
$\frac{\text{odd}}{\text{odd}}$. The key ingredient in the study of Riemann's function was its relation to the theta function 
$\theta(z)=\sum_{n\in\Z}e^{i\pi n^2 z}$, which is an automorphic form of weight $\frac{1}{2}$ under the action of $\theta$-modular group. 
The function $\theta$ appears in the study of continued fractions. (For example Kraaikamp and Lopes in \cite{KL} establish the relation between the 
$\theta$ group and continued fraction with even partial quotients. See Rivoal and Seuret \cite{RS} for an elaboration of this connection for functions 
similar to $S(x)$.) It appears that the pointwise regularity at irrational points is also connected to continued fraction expansions. Let 
$x\in\R\setminus\Q$, and $(a_n)_n$ be the sequence of partial quotients of $x$, that is $x=[a_0;a_1, a_2, ...]$. Let $(\frac{p_n}{q_n})_n$ be the 
sequence of continued fraction approximations of $x$, that is $\frac{p_n}{q_n}=[a_0;a_1, a_2, ..., a_n]$. The convergents can be obtained from partial 
quotients by the recurrence relations: ${p_{n}=a_{n}p_{n-1}+p_{n-2}}$, ${q_{n}=a_{n}q_{n-1}+q_{n-2}}$, for $n\geq0$, and ${p_{-1}=1}, {p_{-2}=0}, 
{q_{-1}=0}, {q_{-2}=1}$. For each $n$, we define $\kappa_n$ by the equality ${\big|x-\frac{p_n}{q_n}\big|= \frac{1}{q_n^{\kappa_n}}}$. We then define
\begin{align*}
\mu(x)&=\limsup_{n\to\infty}\kappa_n,\\
\nu(x)&=\liminf_{n\to\infty}\kappa_n.
\end{align*}
For all $x\in\R\setminus\Q$, we have $\mu(x)\geq\nu(x)\geq2$, and for almost all $x$, $\nu(x)=\mu(x)=2$. 
Let $\mu_{\text{e}}(x)=\limsup_{n\to\infty}\{\kappa_n|p_n,q_n \text{ are not both odd}\}$. 
Using the tools of wavelet analysis, Jaffard proved in 1996 in \cite{J2} that the H\"{o}lder regularity exponent of $S$ at an irrational point $x$ is equal to
$$\alpha(x)=\frac{1}{2}+\frac{1}{2\mu_{\text{e}}(x)}.$$
In our analysis we follow the method proposed by Jaffard. However, before we state our results, we would like to stress that we allow $\alpha\in\N$. We just bear in mind that if we write that $\alpha(x_0)=\alpha$ for $\alpha\in\N$, we do not mean that the function is $\alpha$ times differentiable at $x_0$.
For instance $x\mapsto x\log(x)$ has H\"{o}lder exponent 1 at $x=0$, but it is not differentiable there.

\medskip
Before we state our results, we mention that if $\nu(x)=\infty$ or $\mu(x)=\infty$, we use the convention that $\frac{1}{\nu(x)}=0$ or $\frac{1}{\mu(x)}=0$ and all our theorems remain valid in this case. 
Let $k\geq4$ be even. The series $M_{k,s}$ converges normally on $\R$ for all $s>k$.
We prove this fact (and more) in Section~\ref{swavelets}.

\begin{thm}\label{thm:MknotcuspHE}
 Let $k\geq 4$, even, and $M_k$ be a modular form of weight $k$ under $SL_2(\Z)$ not a cusp form. For $x\in \R\setminus\Q$, let $\alpha_{k,s}(x)$ be the H\"{o}lder regularity exponent of $M_{k,s}$ at $x$. Assume that
\begin{equation}\label{labelno1}
s>k+\frac{k}{\nu(x)}-\frac{k}{\mu(x)}.
\end{equation}
Then
$$\alpha_{k,s}(x) = s-k+\frac{k}{\mu(x)}.$$
\end{thm}
\begin{rmk}
We note that, if $s>\frac{3k}{2}$, then (\ref{labelno1}) is satisfied for all $x\in\R\setminus\Q$. We do not know if  (\ref{labelno1}) can be relaxed to $s>k$ for any $x\in\R\setminus\Q$. 
However, it is satisfied for almost all $x$ for any $s>k$. 
\end{rmk}



In the proof of Theorem~\ref{thm:MknotcuspHE} we use the fact that if $M_k$ is not a cusp form, then $|M_k(z)|$ is bounded below by a positive constant when $\im(z)\to\infty$.  
A cusp form $M_k$ does not have this property, therefore in this case we have a weaker version of Theorem~\ref{thm:MknotcuspHE}, namely the following.
\begin{thm}\label{mainthm3}
 Let $k\geq 4$, even, and $M_k$ be a cusp form of weight $k$ under $SL_2(\Z)$. For $x\in \R\setminus\Q$, let $\beta_{k,s}(x)$ be the H\"{o}lder regularity exponent of $M_{k,s}$ at $x$. Assume that
\begin{equation}\label{labelno2}
s>\frac{k}{2}+1+\frac{2}{\nu(x)}-\frac{2}{\mu(x)}.
\end{equation}
\begin{enumerate}[(i)]
 \item We have
$$\beta_{k,s}(x) \geq s-\frac{k}{2}-1+\frac{2}{\mu(x)}.$$
\item Moreover, if there exists $N\in\N$ such that for infinitely many $n$
\begin{equation}\label{an15}
 a_{n}(x)=N,
\end{equation}
and if $\mu(x)=2,$ then
$$\beta_{k,s}(x) = s-\frac{k}{2}.$$
\end{enumerate}
\end{thm}

\begin{rmk}
For all $s>\frac{k}{2}+1$ this condition 
is satisfied for almost all $x$. 
Let $\pi_i(x,n)=\frac{1}{n}|\{1\leq j\leq n|a_j=i \}|$ denote the frequency of appearance of $i$ among the first $n$ partial quotients of $x$. It is well-known that for almost all $x$ we have $\lim_{n\to\infty} \pi_i(x,n)= \frac{1}{\log(2)}\log(1+ \frac{1}{i(i+2)})$, see \cite[p. 225]{IK}. In particular, Condition~(\ref{an15}) is also satisfied for almost all $x$.
\end{rmk}

Theorems \ref{thm:MknotcuspHE} and \ref{mainthm3} remain valid if we replace the sine series $M_{k,s}(x)=\sum_{n=1}^{\infty}\frac{r_n}{n^s}\sin(2\pi nx)$ with the cosine series $\sum_{n=1}^{\infty}\frac{r_n}{n^s}\cos(2\pi nx)$.

\medskip

Our method does not enable us to compute H\"{o}lder regularity exponents at rational points, and if the H\"{o}lder regularity exponent at $x$ is 
a natural number $\alpha$, we do not know if the function is actually $\alpha$ times differentiable at $x$. However, the approach presented by Itatsu 
in \cite{I} seems to give some complementary information, see \cite{P}. 

More information about the local behaviour of $M_{k,s}$ could be obtained by the study of its local oscillations, {\em chirps}-like behaviours (see \cite{J1, JM}). Also, further study could include considering two-microlocal spaces $C^{\alpha,\alpha'}$ instead of $C^\alpha$ (see \cite{JM, O}).
Differentiability and H\"{o}lder regularity of series of this type was also studied by Chamizo in \cite{Ch}. In this paper, he studied the series arising from automorphic forms $f(x)=\sum_{n=0}^{\infty}r_ne^{2\pi inx}$ of positive weights $k$ under a Fuchsian group with a multiplier system: $f_s(x)=\sum_{n=1}^{\infty}\frac{r_n}{n^s}e^{2\pi inx}$. His method is based on the theory of automorphic forms. Assuming that $f$ is a cusp form, Chamizo proved that $f_{s}$ is not differentiable at any irrational $x$ if $s<\frac{k}{2}+1$, and if $\frac{k+1}{2}<s<\frac{k}{2}+1$, then $f_{s}$ is differentiable at all rational points. Moreover, it follows from \cite[Theorem~2.1]{Ch} that the H\"{o}lder regularity exponent of $f_s$ at irrational points is equal to $s-\frac{k}{2}$ for all $\frac{k}{2}<s<\frac{k}{2}+1$.

\section{Wavelet transform}\label{swavelets}

\subsection{Wavelets and regularity of functions}

We define the transform of an $L^{\infty}$ function $f$ with respect to the wavelet $\psi\in L^1(\R)$ as follows:
$$C(a,b)(f)=\frac{1}{a}\int_{\R}f(t)\overline{\psi}\left(\frac{t-b}{a}\right)dt,$$
where $\overline{\psi}$ denotes the complex conjugate of $\psi$, $a>0$, and $b\in\R$.
On the other hand, we can reconstruct the function from its wavelet transform, using the formula:
$$f(t)=\int_{0}^\infty \frac{da}{a^2}\int_\R g\left(\frac{t-b}{a}\right)C(a,b)(f)db,$$
where $g$ is a reconstruction wavelet. A reconstruction wavelet is a function that depends on $\psi$, but it is not unique, in some cases we can have $g=\psi$, the conditions which $g$ must satisfy are given in \cite[(2.1)]{HT}. In the last 20 years, it has been established that wavelets, which originate from applied mathematics, can be very useful in the analysis of pointwise regularity. Apart from the paper by Holschneider and Tchmitchian \cite{HT}, we should mention monographs by St\'{e}phane Jaffard and Yves Meyer \cite{JM}, \cite{M} in which they describe in detail the connection between wavelets and regularity. Also Oppenheim in his thesis \cite{O} applied wavelet theory in his study of regularity of a two-dimensional analogue of Riemann series (\ref{riems}). 
For background information about wavelets, we refer the reader to the book by Ingrid Daubechies, ``Ten Lectures on Wavelets'' \cite{Da}, chapter 2 is especially relevant for this paper. 

We will denote the Fourier Transform of a function $g$ by $\hat{g}(\xi)=\int_\R g(x)e^{-ix\xi}dx$.
We now recall Proposition~1 from \cite{J1}. 

\begin{prop*}
Let $\alpha>0$, and $m=[\alpha]$ its integer part. Assume the following:
\begin{enumerate}
 \item $|\psi(x)|+|\psi^{(1)}(x)|+...+|\psi^{(m+1)}(x)|\leq \frac{c}{(1+|x|)^{m+2}}$, for some constant $c$ that may depend on $m$ only;
 \item $\int_\R \psi(x)dx=\int_\R x\psi(x)dx=...=\int_\R x^m\psi(x)dx=0$;
 \item $\hat{\psi}(\xi)=0$ if $\xi<0$;
 \item $\int_0^\infty |\hat{\psi}(\xi)|^2\frac{d\xi}{\xi}<\infty$.
\end{enumerate}
Let $a\in(0,1)$, $b\in\R$. If $f :  \R\to\R \in C^{\alpha}(x_0)$, then for some $C$ that depends at most on $x_0$ and $f$, we have
$$|C(a,b)(f)|\leq Ca^{\alpha}\left(1+\frac{|b-x_0|}{a}\right)^{\alpha}.$$
Conversely, if for some $C$ that depends at most on $x_0$ and $f$ we have
$$|C(a,b)(f)|\leq Ca^{\alpha}\left(1+\frac{|b-x_0|}{a}\right)^{\alpha'} \textnormal{ for an } \alpha'<\alpha,$$
as $b\to x_0$ and $a\to 0$, then $f \in C^{\alpha}(x_0)$.
\end{prop*}

\subsection{The wavelet $\psi_s$}\label{subs22}
In this paper, we will work with the principal branch $-\pi <\arg(z) \leq \pi$ of $z\in\C$. 
For $s>0$ and $x\in\R$, consider $$\psi_s(x)=\frac{1}{(x+i)^{s+1}}.$$ 
We now show that $\psi_s$ satisfy the assumptions 1-4 of Proposition~J.
We start by noting the following facts that will be used later.

\begin{lem}\label{sint}
 Let $\rho>0$, $z\in \C\setminus\R$. We have
$$ \int_\R \frac{e^{it}}{(t-z)^\rho}dt=\left\{
     \begin{array}{lr}
       \widetilde{c}(\rho)e^{i z} & \text{if} \quad \im( z)>0, \\
       0 & \text{if} \quad \im(z)<0,
     \end{array}
   \right. $$
with $\widetilde{c}(\rho)=\frac{2\pi e^{i\pi \rho/2}}{\Gamma(\rho)}$.
\end{lem}
\begin{proof}
First assume that $\im(z)>0$. Then we have $\int_\R \frac{e^{it}}{(t-z)^\rho}dt=i^{\rho}\int_\R \frac{e^{it}}{(-iz+it)^\rho}dt.$ The result follows from \cite[Equation 6, p. 347]{GR} with $p=-1$, $\nu=\rho$ and $\beta=-iz$. On the other hand, if $\im(z)<0$, then $\int_\R \frac{e^{it}}{(t-z)^\rho}dt=e^{-i\pi\rho}i^{\rho}\int_\R \frac{e^{it}}{(iz-it)^\rho}dt.$ The result then follows from \cite[Equation 7, p. 347]{GR} with $p=-1$, $\nu=\rho$ and $\beta=iz$.
\end{proof}

Then we calculate the Fourier transform of $\psi_s$.
\begin{lem}\label{lem:ftpsis}
 For $s >1$, we have
\begin{equation*}
 \hat{\psi}_s(\xi)=
\begin{cases}
 e^{-i\pi(s+1)}\xi^s \widetilde{c}(s+1) e^{-\xi} & \text{if} \quad \xi>0,\\
 0 & \text{if} \quad \xi<0.
\end{cases}
\end{equation*}
\end{lem}
\begin{proof}
  By definition of the Fourier Transform, we have
\begin{equation*}
 \hat{\psi}_s(\xi)=\int_\R \frac{e^{-ix\xi}}{(x+i)^{s+1}}dx
=\begin{cases}
 e^{-i\pi(s+1)}\xi^s \int_\R \frac{e^{it}}{(t-i\xi)^{s+1}}dt & \text{if} \quad \xi>0,\\
 -e^{i\pi(s+1)}\xi^s \int_\R \frac{e^{it}}{(t-i\xi)^{s+1}}dt & \text{if} \quad \xi<0.
\end{cases}
\end{equation*}
We conclude by Lemma~\ref{sint}.
\end{proof}

\noindent$\bullet$ \textit{Assumption 1}

\begin{lem}\label{lem1}
 For all $k\in\N^*$ even, $s>k$ and $x \in \R$, there exists $\delta_0>0$ such that for all $0<\delta\leq \delta_0$ we have
 $$ |\psi_s(x)|+\left|\psi_s^{(1)}(x)\right|+...+\left|\psi_s^{(m+1)}(x)\right| \leq \frac{c}{(|x|+1)^{m+2}},$$
with $m=\left[s-k+\frac{k}{\mu(x)-\delta}\right]$, for some constant $c$.
\end{lem}
\begin{proof}
For all $x \in \R$, from $ (|x|-1)^2 \geq 0$ we get
\begin{equation*}
\frac{2^{1/2}}{|x|+1} \geq \frac{1}{(|x|^2+1)^{1/2}}= \frac{1}{|x+i|}.
\end{equation*}
Then we note that for all $n\in\N^*$ we have
$$\psi_s^{(n)}(x)=\frac{(s+1)(s+2)...(s+n)}{(x+i)^{s+1+n}}.$$
If $\delta_0\leq 1$, then $s+1+n\leq m+2$, we have
$$|\psi_s^{(n)}(x)|\leq \frac{c}{|x+1|)^{m+2}},$$
for all $n\in\N^*$ for some constant $c$. It suffices to show now that $ |\psi_s(x)| \leq \frac{1}{(|x|+1)^{m+2}}.$ 

Let $\delta_0<\frac{k\mu(x)-k-\{s\}\mu(x)}{k - \{s\}}$. Since $k-\{s\}>0$ and $k\mu(x)-k-\{s\}\mu(x)>0$ for all $x\in\R, s>k$ and $k\geq 2$, we have $\delta_0>0$. It follows that $\{s\}+\frac{(1-k)(\mu(x)-\delta_0)+k}{\mu(x)-\delta_0}<1$, therefore $ \left[s-k+\frac{k}{\mu(x)-\delta}\right] \leq s-1$ for all $\delta\leq \delta_0$, which completes the proof of the Lemma.
\end{proof}

\noindent$\bullet$ \textit{Assumption 2}

\begin{lem}\label{lem2}
 For $s >1$ and $\alpha<s$, we have
 $$\int_\R \psi_s(x)dx=\int_\R x\psi_s(x)dx=...=\int_\R x^m\psi_s(x)dx=0,$$
with $m=[\alpha]$.
\end{lem}
\begin{proof}
Set $\hat{\psi}_s(0)=0$, since $s>0$, by Lemma~\ref{lem:ftpsis} $\hat{\psi}_s$ is a continuous function. Then for all $n<s$, we have
$$\hat{\psi}^{(n)}_s(\xi)= \int_\R \frac{(-ix)^n e^{-ix\xi}}{(x+i)^{s+1}}dx
=
\begin{cases}
 \widetilde{c}(n,s)(\xi^s e^{-\xi})^{(n)} & \text{if} \quad \xi\geq 0,\\
 0 & \text{if} \quad \xi\leq0,
\end{cases}$$
for some constant $\widetilde{c}(n,s)$. In particular, the function is $0$ at $\xi=0$. As $\alpha<s$, it follows that for all $n\leq m$, we have
$$\int_\R \frac{x^n}{(x+i)^{s+1}}dx=0.$$
\end{proof}

\noindent$\bullet$ \textit{Assumption 3}

\begin{lem}\label{lem3}
 Let $\hat{\psi}_s$ be the Fourier transform of $\psi_s$.
If $\xi <0$ then
$$\hat{\psi}_s(\xi)=0.$$
\end{lem}
\begin{proof}
It follows from the first part of the proof of Lemma~\ref{lem:ftpsis}.
\end{proof}

\noindent$\bullet$ \textit{Assumption 4}

\begin{lem}\label{lem4}
 We have
$$\int_0^{\infty} |\hat{\psi}_s(\xi)|^2\frac{d\xi}{\xi}<\infty.$$
\end{lem}
\begin{proof}
By Lemma~\ref{lem:ftpsis} we have
\begin{equation*}
 \int_0^{\infty} |\hat{\psi}_s(\xi)|^2\frac{d\xi}{\xi}= |\widetilde{c}(s+1)|^2\int_0^{\infty} \xi^{2s-1} e^{-2\xi}d\xi
=|\widetilde{c}(s+1)|^2 2^{-2s}\Gamma(2s)<\infty.
\end{equation*}
\end{proof}

We have shown that $\psi_s$ fulfils the assumptions 1-4 of Proposition~J.

\subsection{Wavelet transform of $M_{k,s}$}

Before we calculate the wavelet transform, we will show the convergence of $M_{k,s}$ for certain~$s$.
\begin{lem}\label{lem:Mksconv}
 The series $M_{k,s}$ converges normally on $\R$ for all $s>k$. Moreover, if $M_k$ is a cusp form, then $M_{k,s}$ is well-defined and continuous on $\R$ for all $s>\frac{k}{2}$.
\end{lem}
\begin{proof}
By Hecke, if $M_k$ is not a cusp form, then $r_n=O(n^{k-1})$, which proves the first part of the Lemma. Then Deligne proved that if $M_k$ is a cusp form, then $r_n=O(n^{(k-1)/2+\varepsilon})$, for all $\varepsilon>0$. For details, see for example \cite[p.~153-154]{S}. Chamizo improved Deligne's result, showing that if $M_k$ is a cusp form, then $M_{k,s}$ is well-defined and continuous on $\R$ for all $s>\frac{k}{2}$, \cite[Proposition 3.1]{Ch}.
\end{proof}

We also need the following fact in order to calculate the wavelet transform of $M_{k,s}$.

\begin{lem}\label{intsin}
 Let $\rho>1$, then
$$\int_\R \frac{\sin(t)}{(t-z)^\rho}dt=
\begin{cases}
  \frac{\pi e^{i\pi (\rho-1)/2}}{\Gamma(\rho)}e^{iz} & \text{if} \quad \im(z)>0,\\
   \frac{\pi e^{-i\pi (\rho-1)/2}}{\Gamma(\rho)}e^{-iz} & \text{if} \quad \im(z)<0.
\end{cases}$$
\end{lem}
\begin{proof}
 Recall that by Lemma~\ref{sint}, we have
$$ \int_\R \frac{e^{it}}{(t-z)^\rho}dt=\left\{
     \begin{array}{lr}
       \widetilde{c}(\rho)e^{i z} & \text{if} \quad \im( z)>0, \\
       0 & \text{if} \quad \im(z)<0.
     \end{array}
   \right. $$
On the other hand
\begin{equation*}
 \int_\R \frac{e^{it}}{(t-z)^\rho}dt=\int_\R \frac{e^{-iu}}{(-u-z)^\rho}du=\left\{
     \begin{array}{lr}
       e^{i\pi \rho} \int_\R \frac{e^{-iu}}{(u+z)^\rho}du & \text{if} \quad \im( z)>0, \\
       e^{-i\pi \rho}\int_\R \frac{e^{-iu}}{(u+z)^\rho}du & \text{if} \quad \im(z)<0.
     \end{array}
   \right. 
\end{equation*}
thus
$$ \int_\R \frac{e^{-iu}}{(u-z)^\rho}du=\left\{
     \begin{array}{lr}
       0 & \text{if} \quad \im(z)>0,\\
      \widetilde{c}(\rho)e^{-i\pi \rho} e^{-i z} & \text{if} \quad \im(z)<0.
     \end{array}
   \right. $$
The result then follows from $\sin(u)=\frac{e^{iu}-e^{-iu}}{2i}$.
\end{proof}

Now we will calculate the wavelet  coefficients of $M_{k,s}$ with respect to the wavelet $\psi_s$. 
\begin{lem}\label{lem:wavtmk}
 The wavelet transform of $M_{k,s}$ with respect to the wavelet $\psi_s$ is
$$C_{k,s}(a,b):=C(a,b)(M_{k,s})=\widehat{C} a^s (M_k(b+ia)-r_0),$$
where $\widehat{C}=(2\pi)^s\frac{\pi e^{i\pi s/2}}{\Gamma(s+1)}$. In particular, if $M_k$ is a cusp form, then
$$C_{k,s}(a,b)=\widehat{C}a^s M_k(b+ia).$$
\end{lem}
\begin{proof}
We have
\begin{equation*}
  C(a,b)(M_{k,s})=\frac{1}{a}\int_\R M_{k,s}(x)\frac{1}{\left(\frac{x-b}{a}-i\right)^{s+1}}dx
=\frac{1}{a}\sum_{n=1}^{\infty}\frac{r_n}{n^s}\int_\R \frac{\sin(2\pi nx)}{\left(\frac{x-b}{a}-i\right)^{s+1}}dx.
\end{equation*}
Then we use the substitution $u=2\pi n x$, and we obtain 
\begin{align*}
 C(a,b)(M_{k,s})
&=\frac{1}{a}\sum_{n=1}^{\infty}\frac{r_n}{n^s}\int_\R \frac{\sin(u)}{\left(\frac{\frac{u}{2\pi n}-b}{a}-i\right)^{s+1}} \frac{du}{2\pi n}\\
&= \frac{1}{a}\sum_{n=1}^{\infty}\frac{r_n}{n^s}\int_\R \frac{a^{s+1}(2\pi)^{s+1}n^{s+1}\sin(u)}{\left(u-2\pi n b-2\pi n ai\right)^{s+1}}\frac{du}{2\pi n}\\
&= \frac{1}{a}\sum_{n=1}^{\infty}\frac{r_n}{n^s}\int_\R \frac{a^{s+1}(2\pi)^{s+1}n^{s+1}\sin(u)}{\left(u-2\pi n b-2\pi n ai\right)^{s+1}}\frac{du}{2\pi n}\\
&= (2\pi)^{s}a^s\sum_{n=1}^{\infty}r_n\int_\R \frac{\sin(u)}{\left(u-2\pi n b-2\pi n ai\right)^{s+1}}du.
\end{align*}
Then, by Lemma~\ref{intsin} we have
$$ C(a,b)(M_{k,s})
=a^s(2\pi)^s\widetilde{c}(s+1)\sum_{n=1}^{\infty}r_n e^{2\pi i n(b+ia)}=\widehat{C} a^s (M_{k}(b+ia)-r_0),$$
with $\widehat{C}=(2\pi)^s\widetilde{c}(s+1)=(2\pi)^s\frac{\pi e^{i\pi s/2}}{\Gamma(s+1)}$.
\end{proof}

\section{Estimating $C_{k,s}(a,b)$ if $M_k$ is not a cusp form}\label{sestCknoncusp}

We first estimate $|M_k(z)|$.

\begin{claim}\label{claim:sizeMknotcusp}
Let $M_k$ be a modular form, not a cusp form. There exist $r,c_1, c_2, c_3 >0$, such that:

if $\im(z)\leq r$, then $$|M_k(z)|\leq\frac{c_1}{\im(z)^{k}};$$

if $\im(z) \geq r$, then $$ c_2 \leq|M_k(z)|\leq c_3.$$
\end{claim}
\begin{proof}
Let $r>0$ such that $|r_0|> \sum_{n=1}^\infty |r_n|e^{-2\pi n r}$. Again, by Hecke we have $r_n=O(n^{k-1})$ (see \cite[p.~153-154]{S}). Therefore, there exists $c_{1,k}>0$ such that $|M_k(z)|\leq |r_0|+c_{1,k}\sum_{n=1}^{\infty}n^{k-1}e^{-2\pi  n\im(z)}$. Then there exists a polynomial $P_{k-1}$ of degree ${k-1}$ vanishing at $0$ such that $\sum_{n=1}^{\infty}n^{k-1}e^{-2\pi n \im(z)}=\frac{P_{k-1}(e^{-2\pi \im(z)})}{(1-e^{-2\pi \im(z)})^k}$. Since $0< e^{-2\pi \im(z)}<1$, there exists $c_{2,k}>0$ such that $|P_{k-1}(e^{-2\pi \im(z)})|\leq c_{2,k}e^{-2\pi \im(z)}$. Finally, there exists $c_{3,k}>0$ such that $\frac{e^{-2\pi \im(z)}}{(1-e^{-2\pi \im(z)})^k}\leq \frac{c_{3,k}}{(2\pi \im(z))^k}$. Summing up, we get that
$$|M_k(z)|\leq |r_0|+\frac{c_{1,k}c_{2,k}c_{3,k}}{(2\pi)^k}\frac{1}{\im(z)^k}.$$
If $\im(z)\leq r$, then the first part of the Claim follows from setting $c_1= |r_0|r^k+\frac{c_{1,k}c_{2,k}c_{3,k}}{(2\pi)^k}$.

On the other hand, if $\im(z)\geq r$, then letting $c_2 = |r_0| - \sum_{n=1}^\infty |r_n|e^{-2\pi n r}$ and $c_3=|r_0|+ \sum_{n=1}^\infty |r_n|e^{-2\pi n r}$ gives the result.
\end{proof}

The following proposition is an analogue of Proposition~2 in \cite{J2}. The significant difference is that Jaffard fixes $D=3$. This is possible because in the analogue of Claim~\ref{claim:sizeMknotcusp} he can take $r=1$. We cannot do it in general. In order to be able to use the lower bound from Claim~\ref{claim:sizeMknotcusp}, we need to carefully choose $D$, as we will see in the proof of the Proposition.

\begin{prop}\label{propkk}
 Let $x\in\R\setminus\Q$. Let $a\in(0,1), b\in \R$.  
\begin{enumerate}[(i)]
 \item Let $D>1$. For each $n$, if
\begin{equation}\label{condonab}
 D\left|x-\frac{p_n}{q_n}\right| \leq |b-x+ia| \leq D\left|x-\frac{p_{n-1}}{q_{n-1}}\right|,
\end{equation}
we have either:
$$ |C_{k,s}(a,b)|\leq C a^{s-k+k/\kappa_{n-1}}\left(1+\frac{|b-x|}{a}\right)^{k/\kappa_{n-1}},$$
or
\begin{equation*}
 |C_{k,s}(a,b)|\leq C a^{s-k+k/\kappa_{n}}\left(1+\frac{|b-x|}{a}\right)^{k/\kappa_{n}},
\end{equation*}
for a constant $C$ that may depend on $k$, $s$, $x$ and $D$.
\item There exists $D_0>1$ depending at most on $k$, $s$ and $x$, and there exists $\widetilde{C}>0$ that may depend on $k$, $s$, $x$ and $D_0$ such that for infinitely many $n$,
 there exists a point $b+ia$ in the domain (\ref{condonab}) with $D=D_0$ satisfying
$$ |C_{k,s}(a,b)|\geq \widetilde{C} a^{s-k+k/\kappa_{n}}\left(1+\frac{|b-x|}{a}\right)^{k/\kappa_{n}}.$$ 
\end{enumerate}
\end{prop}
\begin{proof}
For a matrix $\gamma=\bigl(\begin{smallmatrix}
a&b\\ c&d
\end{smallmatrix} \bigr) \in SL_2(\Z)$, and $z\in\C$ we will denote the fraction transformation as
$$\gamma\cdot z= \frac{az+b}{cz+d},$$
if $cz+d\in\C\setminus\{0\}$, and 
$ \gamma\cdot \left(-\frac{d}{c}\right)= \infty.$

Consider 
$$\gamma_n= \left(\begin{matrix}
(-1)^nq_{n-1}&(-1)^{n-1}p_{n-1}\\ q_n&-p_n
\end{matrix} \right).$$
We have
$(-1)^nq_{n-1}(-p_n)-(-1)^{n-1}p_{n-1}q_n=(-1)^{n-1}(q_{n-1}p_n-p_{n-1}q_n)=1$, which shows that $\gamma_n\in SL_2(\Z)$. 
Let $z\in\C$ with $\im(z)>0$. We have
\begin{align*}
 \gamma_n\cdot\left(\frac{p_n}{q_n}+z\right)&=\frac{(-1)^nq_{n-1}\big(\frac{p_n}{q_n}+z\big)+(-1)^{n-1}p_{n-1}}{q_n\big(\frac{p_n}{q_n}+z\big)-p_n}\\
&= \frac{(-1)^n(p_nq_{n-1}-p_{n-1}q_n)+z(-1)^nq_{n-1}q_n}{q_n^2 z}\\
&=\frac{-1+(-1)^nzq_{n-1}q_n}{q_n^2 z}\\
&=\frac{(-1)^nq_{n-1}}{q_n}-\frac{1}{q_n^2 z}.
\end{align*}

We recall that for any $\gamma=\bigl(\begin{smallmatrix}
a&b\\ c&d
\end{smallmatrix} \bigr) \in SL_2(\Z)$, we have
\begin{equation}\label{eqationmod}
 M_k(z)=\frac{M_k(\gamma\cdot z)}{(cz+d)^k}.
\end{equation}
Then we have
\begin{equation}\label{mod} 
\left|M_k\left(\frac{p_n}{q_n}+z\right)\right|=\frac{\left|M_k\left(\gamma_n\cdot\left(\frac{p_n}{q_n}+z\right)\right)\right|}{\left|q_n\left(\frac{p_n}{q_n}+z\right)-p_n\right|^k}
= \frac{\left|M_k\left(\frac{(-1)^nq_{n-1}}{q_n}-\frac{1}{q_n^2 z}\right)\right|}{|q_n z|^k}.
\end{equation}
Then we observe that $$\im\left(\frac{(-1)^nq_{n-1}}{q_n}-\frac{1}{q_n^2 z}\right)=\frac{\im(z)}{q_n^2 |z|^2}.$$ We now consider two cases.

\medskip
\noindent\textbf{\textit{Case 1:}} Assume that $\frac{\im(z)}{q_n^2 |z|^2} \leq r$. By Claim \ref{claim:sizeMknotcusp} we have
$$\left|M_k\left(\frac{p_n}{q_n}+z\right)\right|\leq c_1\frac{|q_n^2 z^2|^{k}}{|q_n z|^k\im(z)^{k}}
= c_1\frac{q_n^{k} |z|^{k}}{\im(z)^{k}}.$$
For $z=b+ia-\frac{p_n}{q_n}$, we have
\begin{equation*}
  \left|M_k\left(b+ia\right)\right| \leq c_1a^{-k}q_n^{k} \left|b+ia-\frac{p_n}{q_n} \right|^{k}.
\end{equation*}
By (\ref{condonab}), we have
\begin{equation}\label{eq:zitoab}
  \frac{D-1}{D}|b+ia-x|
\leq \left|b+ia-\frac{p_n}{q_n} \right|
\leq \frac{D+1}{D}|b+ia-x|.
\end{equation}
Also by (\ref{condonab}), since
$$\frac{D}{q_n^{\kappa_n}}\leq|b+ia-x|\leq\frac{D}{q_{n-1}^{\kappa_{n-1}}}$$
and $$\frac{1}{q_{n-1}^{\kappa_{n-1}-1}}\leq\frac{1}{q_n},$$
noting that $D>1$, we have
\begin{equation}\label{eq:qnitoab}
 |b+ia-x|^{-1/\kappa_n}\leq q_n \leq D|b+ia-x|^{(1-\kappa_{n-1})/\kappa_{n-1}}.
\end{equation}
Substituting it, we get
\begin{align}
  \left|M_k\left(b+ia\right)\right| 
&\leq c_1(D+1)^{k}a^{-k}|b+ia-x|^{k/\kappa_{n-1}} \notag\\
&\leq c_1(D+1)^{k}a^{-k+k/\kappa_{n-1}}\left(1+\frac{|b-x|}{a}\right)^{k/\kappa_{n-1}} .\label{Ekeq1}
\end{align}
Since $-k+k/\kappa_{n-1}<0$ and $a\in(0,1)$, we have $a^{-k+k/\kappa_{n-1}}>1$. Also, as $k/\kappa_{n-1}>0$, we have $\left(1+\frac{|x-b|}{a}\right)^{k/\kappa_{n-1}}>1$. Then
\begin{align*}
|C_{k,s}(a,b)|&=|a^sc_k(s)(1-M_k(b+ia))|\leq a^s|c_k(s)|(1+|M_k(b+ia)|)\\
&\leq a^s|c_k(s)|\left(1+c_1(D+1)^{k}a^{-k+k/\kappa_{n-1}}\left(1+\frac{|b-x|}{a}\right)^{k/\kappa_{n-1}}\right)\\
&\leq a^s|c_k(s)|\left(1+c_1(D+1)^{k}\right)a^{-k+k/\kappa_{n-1}}\left(1+\frac{|b-x|}{a}\right)^{k/\kappa_{n-1}}.
\end{align*}
The result follows with $C=|c_k(s)|\left(1+c_1(D+1)^{k}\right)$.

\medskip
\noindent\textbf{\textit{Case 2:}} Assume that $\frac{\im(z)}{q_n^2 |z|^2} > r$. By Claim \ref{claim:sizeMknotcusp} we have 
$$\left|M_k\left(\frac{p_n}{q_n}+z\right)\right|\leq \frac{c_3}{|q_n z|^k}.$$
By (\ref{eq:qnitoab}) and (\ref{eq:zitoab}), we get
\begin{align}
\left|M_k\left(b+ia\right)\right|
&\leq c_3\left(\frac{D}{D-1}\right)^{k}|b+ia-x|^{-k+k/\kappa_n}\leq c_3\left(\frac{D}{D-1}\right)^{k}a^{-k+k/\kappa_n}\notag\\
&\leq c_3\left(\frac{D}{D-1}\right)^{k}a^{-k+k/\kappa_n}\left(1+\frac{|x-b|}{a}\right)^{k/\kappa_n}.\label{Ekeq2}
\end{align}
As before, since $a^{-k+k/\kappa_n}\left(1+\frac{|x-b|}{a}\right)^{k/\kappa_n}>1$, we have 
\begin{align*}
|C_{k,s}(a,b)|&=|a^sc_k(s)(1-M_k(b+ia))|\leq a^s|c_k(s)|(1+|M_k(b+ia)|)\\
&\leq a^s|c_k(s)|\left(1+c_3\left(\frac{D}{D-1}\right)^{k}a^{-k+k/\kappa_n}\left(1+\frac{|x-b|}{a}\right)^{k/\kappa_n}\right)\\
&\leq a^s|c_k(s)|\left(1+c_3\left(\frac{D}{D-1}\right)^{k}\right)a^{-k+k/\kappa_n}\left(1+\frac{|x-b|}{a}\right)^{k/\kappa_n}.
\end{align*}
The result follows with $C= |c_k(s)|\left(1+c_3\left(\frac{D}{D-1}\right)^{k}\right)$.

\medskip
For the second part of Proposition~\ref{propkk}, first suppose that $(q_n^{\kappa_n-2})_n$ is unbounded. Then for any $D>1$ there exists an increasing sequence $(n_m)_m$, such that for all $m$ we have 
\begin{equation}\label{opt2}
q_{n_m}^{\kappa_{n_m}-2}>\frac{(D^2+1)}{D}r,
\end{equation}
where $r$ is the constant defined in Claim \ref{claim:sizeMknotcusp}, and $n_m$ is large enough so that $q_{n_m}^{k\kappa_{n_m}-k}>4(\sqrt{D^2+1})^k$. 
Now consider the point $a=\frac{D}{q_{n_m}^{\kappa_{n_m}}}, b=x$, which satisfies (\ref{condonab}). Then we see that with 
$z=b+ia-\frac{p_{n_m}}{q_{n_m}}$, 
$$|z|= \frac{\sqrt{D^2+1}}{q_{n_m}^{\kappa_{n_m}}},$$
and it follows that  
$$\frac{\im(z)}{q_n^2 |z|^2}=\frac{D}{D^2+1} q_n^{\kappa_n-2}.$$
Then by Claim \ref{claim:sizeMknotcusp}, we have
$$\left|M_k\left(\frac{(-1)^{n_m}q_{n_m-1}}{q_{n_m}}-\frac{1}{q_{n_m}^2 z}\right)\right|\geq c_2.$$
By (\ref{mod}) we have
\begin{equation}\label{Ekeq3}
 \left|M_k\left(\frac{p_{n_m}}{q_{n_m}}+z\right)\right|
\geq  \frac{c_2}{|q_{n_m} z|^k}=\frac{c_2q_{n_m}^{-k+k\kappa_{n_m}}}{(\sqrt{D^2+1})^k},
\end{equation}
and hence by (\ref{opt2}), we have
\begin{align*}
 |C_{k,s}(a,b)|&\geq c_k(s)\frac{c_2}{2(\sqrt{D^2+1})^k} q_n^{-s\kappa_n+k\kappa_n-k}
\geq \widetilde{C}\left(\frac{D}{q_{n_m}^{\kappa_{n_m}}}\right)^{s-k+k/\kappa_{n_m}}\\
&= \widetilde{C}a^{s-k+k/\kappa_{n_m}}\left(1+\frac{|b-x|}{a}\right)^{k/\kappa_{n_m}},
\end{align*}
with $\widetilde{C}=\frac{c_2 c_k(s)}{2(\sqrt{D^2+1})^k D^{s-k+k/2}}$, and $D_0=D$. 

\medskip
Now consider the second case, namely suppose that $(q_{n}^{\kappa_{n}-2})_n$ is bounded. We will describe how we choose $D_0$. As $(q_{n}^{\kappa_{n}-2})_n$ is bounded, it has a converging subsequence, and the limit $L_0$ is greater than or equal to 1, because $q_n\geq 1$ and $\kappa_n\geq 2$, for all $n$. Then 
\begin{equation}\label{condM0}
 q_{n_\ell}^{\kappa_{n_\ell}-2} \to L_0\geq 1, \text{ as } \ell \to \infty.
\end{equation}
We also observe that $\left(\frac{(-1)^{n_\ell}q_{n_\ell-1}}{q_{n_l}}\right)_\ell$ is bounded, and has a converging subsequence. Suppose
\begin{equation}\label{condM1}
\frac{(-1)^{n_{\ell(m)}}q_{n_{\ell(m)}-1}}{q_{n_{\ell(m)}}}\to L_1.
\end{equation}
Finally, since $(-1)^{n_{\ell(m)}}=1$ for infinitely many $m$ or $(-1)^{n_{\ell(m)}}=-1$ for infinitely many $m$, we may extract a constant subsequence of $(-1)^{n_{\ell(m)}}$. We will thus assume that all the elements are equal to 1, the same arguments apply to the other case. For simplicity we will denote this subsequence $(n_m)_m$.

Since $M_k$ is a holomorphic function in $\uhp$, we can choose $D_0>1$ and $\delta>0$ such that
\begin{equation}\label{choosingd0}
M_k\left(L_1-\frac{1-iD_0}{D_0^2+1}L_0\right)\neq 0,
\end{equation}
and 
\begin{equation}\label{Eknonzero}
 \left|M_k\left(L_1-\frac{1-iD_0}{D_0^2+1}L_0 +\varepsilon\right)\right|>\delta,
\end{equation}
for $\varepsilon$ small enough. Let $\varepsilon_0>0$ such that for all $\varepsilon\leq \varepsilon_0$, (\ref{Eknonzero}) is satisfied.
For each $m$ consider the point
$$a=\frac{D_0}{q_{n_m}^{\kappa_{n_m}}}; b=x,$$
which satisfies (\ref{condonab}). Using the previous notation $z=b+ia-\frac{p_{n_m}}{q_{n_m}}$, we have
\begin{align*}
 \frac{(-1)^{n_m}q_{n_m-1}}{q_{n_m}}-\frac{1}{q_{n_m}^2 z}&=\frac{q_{n_m-1}}{q_{n_m}}-\frac{1}{q_{n_m}^2 \left( \frac{D_0}{q_{n_m}^{\kappa_{n_m}}}i+x-\frac{p_{n_m}}{q_{n_m}}\right)}\\
&=\frac{q_{n_m-1}}{q_{n_m}}-\frac{1}{q_{n_m}^2 \left( \frac{D_0}{q_{n_m}^{\kappa_{n_m}}}i+\frac{(-1)^{n_m}}{q_{n_m}^{\kappa_{n_m}}}\right)}
=\frac{q_{n_m-1}}{q_{n_m}}-\frac{q_{n_m}^{\kappa_{n_m}-2}}{iD_0+1}\\
&=\frac{q_{n_m-1}}{q_{n_m}}-\frac{(1-iD_0)q_{n_m}^{\kappa_{n_m}-2}}{D_0^2+1}\to L_1-\frac{1-iD_0}{D_0^2+1}L_0,
\end{align*}
as $m\to\infty$, by (\ref{condM0}) and (\ref{condM1}). 
Therefore, there exists $L \in\N$ such that for all $m\geq L$ we have
$$\left|\frac{q_{n_m-1}}{q_{n_m}}-\frac{(1-iD_0)q_{n_m}^{\kappa_{n_m}-2}}{D_0^2+1}-L_1+\frac{1-iD_0}{D_0^2+1}L_0\right|<\varepsilon_0,$$
and  $q_{n_m}^{k\kappa_{n_m}-k}>\frac{2\big(\sqrt{D_0^2+1}\big)^k}{\delta}$.
By (\ref{mod}) and (\ref{Eknonzero}) we have
\begin{equation*}
 \left|M_k\left(\frac{p_{n_m}}{q_{n_m}}+z\right)\right|
\geq  \frac{\delta}{|q_{n_m} z|^k}=\frac{\delta q_{n_m}^{-k+k\kappa_{n_m}}}{(\sqrt{D_0^2+1})^k}.
\end{equation*}
Then we have
\begin{align*}
 |C_{k,s}(a,b)|&\geq \frac{c_k(s)\delta}{2(\sqrt{D_0^2+1})^k} q_n^{-s\kappa_n+k\kappa_n-k}\geq \widetilde{C}\left(\frac{D_0}{q_{n_m}^{\kappa_{n_m}}}\right)^{s-k+k/\kappa_{n_m}}\\
&= \widetilde{C}a^{s-k+k/\kappa_{n_m}}\left(1+\frac{|b-x|}{a}\right)^{k/\kappa_{n_m}},
\end{align*}
with $\widetilde{C}=\frac{c_k(s)\delta}{2(\sqrt{D_0^2+1})^k D_0^{s-k+k/2}}$. This completes the proof of the proposition with $D_0$ satisfying~(\ref{choosingd0}).
\end{proof}

\section{Proof of Theorem \ref{thm:MknotcuspHE}}\label{spth1}

\noindent\textit{Proof of Theorem \ref{thm:MknotcuspHE}}.
Let $x\in\R\setminus\Q$, and assume that $s>k+\frac{k}{\nu(x)}-\frac{k}{\mu(x)}$. Let $\delta_0$ as in Lemma~\ref{lem1}. Assume that $\mu(x)<\infty$, a very similar arguments apply to the other case, and therefore we omit the details. There exists $\delta_1>0$ such that for all $0<\delta<\delta_1$ we have
\begin{equation}\label{eqpsnmd}
 s>k+\frac{k}{\nu(x)-\delta}-\frac{k}{\mu(x)+\delta}.
\end{equation}
Let $0<\delta<\min(\delta_0,\delta_1)$ be given. Then, there exists $N\in\N$ such that for all $n\geq N$ we have
\begin{equation}\label{eqpnkm}
 \nu(x)-\delta\leq \kappa_n \leq \mu(x)+\delta.
\end{equation}
Let $D>1$ and let $\omega=b+ia \in \uhp$ be such that
\begin{equation*}
 |\omega-x|\leq D\left|x-\frac{p_{N}}{q_{N}}\right|.
\end{equation*}
Then we observe that (\ref{condonab}) define half-rings around $x$ (see Figure \ref{fdomain}), and there exists $n_\omega> N$ such that 
\begin{equation*}
  D\left|x-\frac{p_{n_\omega}}{q_{n_\omega}}\right|\leq |\omega-x|\leq D\left|x-\frac{p_{n_\omega-1}}{q_{n_\omega-1}}\right|.
\end{equation*}

By Proposition~\ref{propkk} we have
\begin{equation*}
 |C_{k,s}(a,b)|\leq C a^{s-k+k/\kappa_{n_\omega}}\left(1+\frac{|b-x|}{a}\right)^{k/\kappa_{n_\omega}}
\end{equation*}
or 
\begin{equation*}
 |C_{k,s}(a,b)|\leq C a^{s-k+k/\kappa_{n_\omega-1}}\left(1+\frac{|b-x|}{a}\right)^{k/\kappa_{n_\omega-1}}.
\end{equation*}
It follows from (\ref{eqpnkm}) that
\begin{equation*}
 |C_{k,s}(a,b)|\leq C a^{s-k+k/(\mu(x)+\delta)}\left(1+\frac{|b-x|}{a}\right)^{k/(\nu(x)-\delta)}.
\end{equation*}
Then we conclude by (\ref{eqpsnmd}) and Proposition~J that $M_{k,s} \in C^{s-k+k/(\mu(x)+\delta)}$ at $x$. Letting $\delta \to 0$ shows that $\alpha_{k,s}(x)\geq s-k+\frac{k}{\mu(x)}$. 

\begingroup
\setlength{\medmuskip}{0mu}
\setlength{\thinmuskip}{0mu}
\setlength{\thickmuskip}{0mu}

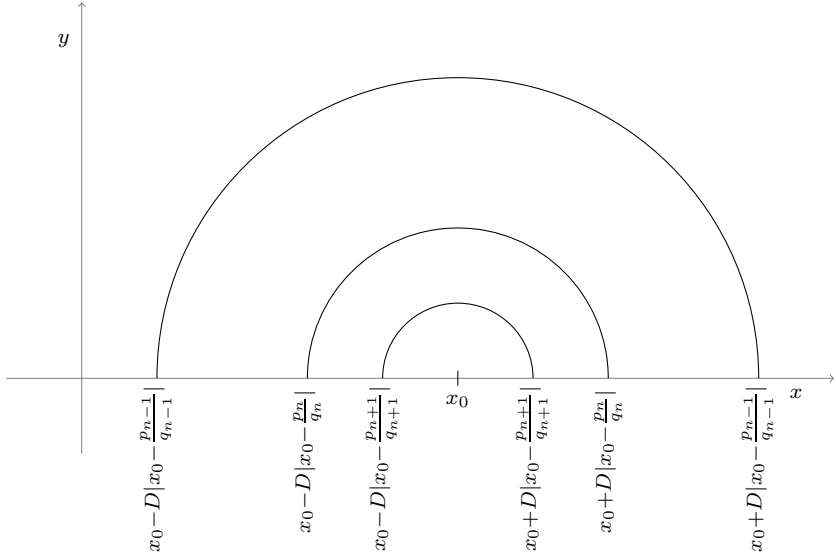
\begin{figure}[h]
\begin{center}
\begin{tikzpicture}
\tikzstyle{every node}=[font=\tiny]

\draw [help lines,->] (-1, 0) -- (10,0);
\draw [help lines,->] (0, -1) -- (0, 5);

\draw (1,0) arc (180:0:4);
\draw (3,0) arc (180:0:2);
\draw (4,0) arc (180:0:1);

\draw (5,-0.1) --(5,0.1);

\node at (9.5,-0.2){$x$};
\node at (-0.24,4.5) {$y$};
\node at (5,-0.3) {$x_0$};
\node [rotate=90] at (4,-1.2) {$x_0-D|x_0-\frac{p_{n+1}}{q_{n+1}}|$};
\node [rotate=90] at (6,-1.2) {$x_0+D|x_0-\frac{p_{n+1}}{q_{n+1}}|$};
\node [rotate=90] at (3,-1.1) {$x_0-D|x_0-\frac{p_{n}}{q_{n}}|$};
\node [rotate=90] at (7,-1.1) {$x_0+D|x_0-\frac{p_{n}}{q_{n}}|$};
\node [rotate=90] at (1,-1.2) {$x_0-D|x_0-\frac{p_{n-1}}{q_{n-1}}|$};
\node [rotate=90] at (9,-1.2) {$x_0+D|x_0-\frac{p_{n-1}}{q_{n-1}}|$};

\end{tikzpicture}
\end{center}
\caption{Half-rings around $x_0$} \label{fdomain}
\end{figure}

\endgroup

\medskip
For the optimality of this exponent, we see that Proposition~\ref{propkk} (ii) implies that for each $\delta>0$ there exists a point $b+ia$, arbitrarily close to $x$ such that
 $$|C_{k,s}(a,b)|\geq \tilde{C} a^{s-k+k/(\mu(x)-\delta)}\left(1+\frac{|b-x|}{a}\right)^{k/(\nu(x)+\delta)}.$$ 
By Proposition~J, we conclude that $M_{k,s}$ is not $C^{s-k+k/(\mu(x)-\delta)}$ at $x$. 
Letting $\delta\to0$ shows that
$$\alpha_{k,s}(x)= s-k+\frac{k}{\mu(x)}.$$
This completes the proof of the theorem. \hfill\qedsymbol

\section{Proof of Theorem \ref{mainthm3}}\label{spth2}

We prove Theorem~\ref{mainthm3} in the same way as Theorem~\ref{thm:MknotcuspHE}. We have the analogues of Claim~\ref{claim:sizeMknotcusp} and Proposition~\ref{propkk}:

\begin{claim}\label{claim3}
Let $M_k$ be a cusp form. There exists $c_1 >0$, such that for all $z\in\uhp$ we have:
 $$|M_k(z)|\leq\frac{c_1}{\im(z)^{k/2+1}}.$$
\end{claim}

\begin{prop}\label{modformprop}
 Let $k\geq 4$, be even. Let $M_k$ be a cusp form of weight $k$, and 
 let $s>\frac{k}{2}+1$.  Let $x\in\R\setminus\Q$. Let $a\in(0,1), b\in \R$. 
\begin{enumerate}[(i)]
 \item Let $D>1$. For each $n$, if
\begin{equation}\label{condonab4}
 D\left|x-\frac{p_n}{q_n}\right| \leq |b-x+ia| \leq D\left|x-\frac{p_{n-1}}{q_{n-1}}\right|,
\end{equation}
we have:
$$ |C_{k,s}(a,b)|\leq C a^{s-k/2-1+2/\kappa_{n-1}}\left(1+\frac{|b-x|}{a}\right)^{2/\kappa_{n-1}},$$
for a constant $C$ that may depend on $k$, $s$ and $x$.
 \item Moreover, let us assume that there exists $N\in \N$ such that for infinitely many $n$,
\begin{equation}\label{an152}
 a_{n}=N.
\end{equation}
Then, there exists $\widetilde{C}>0$ that may depend on $k$, $s$ and $x$, such that for an increasing subsequence of $n$,
there exists a point $b+ia$ in the domain (\ref{condonab4}) with
$$|C_{k,s}(a,b)|\geq \widetilde{C} a^{s-k+k/\kappa_{n-1}}\left(1+\frac{|b-x|}{a}\right)^{k/\kappa_{n-1}}.$$ 
\end{enumerate}
\end{prop}

Their proofs and the proof of Theorem~\ref{mainthm3} are very similar to the proofs in the case on non-cusp form, and therefore ommitted.

\section{Applications}

Let $4\leq k\in\N$ be even. The Eisenstein series of weight $k$ over $\uhp$ is defined as $ E_k(z)=\frac{1}{2\zeta(k)}\sum_{\substack{m,n \in \Z\\ (m,n)\neq (0,0)}} \frac{1}{(m+nz)^k}.$
Its Fourier expansion is
\begin{equation}\label{Ektau}
 E_k(z)=1-\frac{2k}{B_k}\sum_{n=1}^{\infty}\sigma_{k-1}(n)e^{2\pi i nz},
\end{equation}
where $B_k$ is the $k$-th Bernoulli number and $\sigma_{k-1}(n)=\sum_{d|n} d^{k-1}.$ It is modular under the action of $SL_2(\Z)$, but it is not a cusp form.
Define $E_{k,s}(x)= \sum_{n=1}^\infty \frac{\sigma_{k-1}(n)}{n^{s}} \sin(2\pi n x).$ We note that if we let $M_k=E_k$, then $E_{k,s}(x)=-\frac{B_k}{2k}M_{k,s}$, therefore we can apply Theorem~\ref{thm:MknotcuspHE} to it.
Also, Equation~(\ref{Ektau}) defines a quasi-modular function of weight $2$ under the action of $SL_2(\Z)$ when $k=2$. Instead of (\ref{eqationmod}), we have that for all $z\in \uhp$:
\begin{equation*}
 E_2(z)=\frac{E_2(\gamma\cdot z)}{(cz+d)^2}-\frac{6}{i\pi}\frac{c}{(cz+d)}.
\end{equation*}
Because of the addition term $-\frac{6}{i\pi}\frac{c}{(cz+d)}$, we need to add an condition for the optimality of the H\"{o}lder exponent. 

\begin{thm}\label{mainthm1}
For $x\in \R\setminus\Q$, let $\gamma_{2,s}(x)$ be the H\"{o}lder regularity exponent of $E_{2,s}$ at $x$. 
Assume that
\begin{equation}\label{sknumu}
 s>2+\frac{2}{\nu(x)}-\frac{2}{\mu(x)}.
\end{equation}
We have 
$$\gamma_{2,s}(x) \geq s-2+\frac{2}{\mu(x)}.$$
Furthermore, if for infinitely many $n$,
\begin{equation}\label{an7}
 a_{n}(x)\geq 7,
\end{equation}
then
$$\gamma_{2,s}(x)=s-2+\frac{2}{\mu(x)}.$$
\end{thm}

\begin{rmk}
Condition (\ref{an7}) is satisfied for almost all $x$, as the sequence of partial quotients is unbounded for almost all $x$. Condition~(\ref{an7}) is a technical condition and the appearance of $7$ is not significant. It is likely that this condition could be removed.
\end{rmk}

We prove Theorem~\ref{mainthm1} in a very similar way as Theorem~\ref{thm:MknotcuspHE}, therefore we do not present the details of the proof. We note that $|E_2(z)|$ is bounded below by a positive constant as $\im(z)\to \infty$, and in fact it follows from the properties of $\sigma_1$ that the statement of Claim~\ref{claim:sizeMknotcusp} is valid for $M_k=E_2$. The difference in the proof of the theorem is that we also need to treat the additional term arising from quasimodularity.

\medskip

If $s\in\N$ we could apply the methods from \cite{I} to decide how many times $E_{k,s}$ is differentiable at $x\in\R$. It has been done for $E_{2,3}$ 
in \cite{P}, where it has been shown that $E_{2,3}$ is neither differentiable at any rational point, nor at irrational points such that 
$\sum_{n=0}^\infty \frac{\log(q_{n+1})}{q_n^2}=\infty$, but $E_{2,3}$ is differentiable at all $x\in\R\setminus\Q$ such that 
$\sum_{n=0}^\infty \frac{\log(q_{n+1})}{q_n^2} < \infty$, and $\lim_{n\to\infty}\frac{\log(q_{n+4})}{q_n^2}=0$. In particular, if $\mu(x)<\infty$ these two conditions are satisfied. 
We also conjecture that, for any even $k\geq2$, the function $E_{k,k+1}$ is differentiable at $x\in\R\setminus\Q$ if and only if 
\begin{equation}\label{SBcon}
 \sum_{n=0}^\infty \frac{\log(q_{n+1})}{q_n^k} < \infty.
\end{equation}
If $\mu(x)<\infty$ and $\frac{1}{\nu(x)}-\frac{1}{\mu(x)}< \frac{1}{k}$, then Theorem~\ref{thm:MknotcuspHE} implies that for $k\geq 4$ we have 
$\alpha_{k,k+1}(x)=1+\frac{k}{\mu(x)}>1$, and it proves one direction of the conjecture in this case. On the other hand, if $\mu(x)=\infty$ for some 
$x\in\R\setminus\Q$, then by Theorems \ref{thm:MknotcuspHE} and \ref{mainthm1} we conclude that $\alpha_{k,k+1}(x)=1$, for all $k\geq2$, and 
(\ref{SBcon}) could verify whether $E_{k,k+1}$ is differentiable at $x$. 

\bigskip

Consider the discriminant modular form $\Delta$ of weight 12, which can be written
$$\Delta(z)=(2\pi)^{12}\sum_{n=1}^{\infty} \tau(n)e^{2i\pi nz}=\frac{(2\pi)^{12}}{1728}(E_4(z)^3-E_6(z)^2),$$
where $\tau$ is the Ramanujan function. Since $\Delta$ is a cusp form, for any $s>6$ the series
$$\Delta_s(x)=\sum_{n=1}^{\infty} \frac{\tau(n)}{n^s}\cos(2\pi n x)$$
converges for all $x\in\R$. We apply Theorem~\ref{mainthm3} to it. 
\begin{cor}\label{cor:delta}
 For $x\in \R\setminus\Q$, let $\delta_{s}(x)$ be the H\"{o}lder regularity exponent of $\Delta_{s}$ at $x$. Assume that 
$s>7$. Then for almost all $x$ we have
$$\delta_{s}(x) = s-6.$$
\end{cor}
Zagier in \cite{Z} considered series of the type of $\Delta_s$, in particular he studied $\Delta_{11}$ (which he regards as an extension of a quantum modular form) and mentioned that it is 4 times but not 6 times continuously differentiable on $\R$. By Corollary~\ref{cor:delta}, for almost all $x$, we have $\delta_{11}(x)= 5.$

\end{document}